\documentclass[11pt]{amsart}
\usepackage{amssymb,amsmath,amsthm,amstext,amsfonts}
\usepackage{amsmath,amstext,amsthm,amsfonts}
\usepackage[dvips]{graphicx}
\usepackage{epic,eepic}
\usepackage[dvips]{graphicx}
\usepackage{psfrag}
\usepackage{tgtermes} 
\usepackage{color}
\usepackage{mathrsfs}
\usepackage{upgreek}

\makeatletter \@addtoreset{equation}{section} \makeatother

\renewcommand\thetable{\thesection.\@arabic\c@table}

\theoremstyle{plain}
\newtheorem{maintheorem}{Theorem}

\newtheorem{theorem}{Theorem }[section]
\newtheorem{proposition}[theorem]{Proposition}
\newtheorem{lemma}[theorem]{Lemma}
\newtheorem{corollary}[theorem]{Corollary}

\theoremstyle{definition} \theoremstyle{remark}
\newtheorem{remark}[theorem]{Remark}

\newtheorem{definition}[theorem]{Definition}

\DeclareMathAlphabet{\mathpzc}{OT1}{pzc}{m}{it}

\newcommand{\SL}{\text{SL}}

\renewcommand{\epsilon}{\varepsilon}

\begin{document}
\large

\title[A dichotomy in area-preserving reversible maps]{A dichotomy in area-preserving reversible maps}

\author[M. Bessa]{M\'ario Bessa}
\address{Departamento de Matem\'atica, Universidade da Beira Interior, 
Rua Marqu\^es d'\'Avila e Bolama, 
6201-001 Covilh\~a, Portugal.}
\email{bessa@ubi.pt}

\author[A. Rodrigues]{Alexandre Rodrigues}
\address{Departamento de Matem\'atica, Universidade do Porto, 
Rua do Campo Alegre, 687, 
4169-007 Porto, Portugal}
\email{alexandre.rodrigues@fc.up.pt}

\date{\today}

\maketitle

\begin{abstract}
In this paper we study $R$-reversible area-preserving maps $f\colon M\rightarrow M$ on a two-dimensional Riemannian closed manifold $M$, i.e. diffeomorphisms $f$ such that $R\circ f=f^{-1}\circ R$ where $R\colon M\rightarrow M$ is an isometric involution. We obtain a $C^1$-residual subset where any map inside it is Anosov or else has a dense set of elliptic periodic orbits. As a consequence we obtain the proof of the stability conjecture for this class of maps. Along the paper we also derive the $C^1$-closing lemma for reversible maps and other perturbation toolboxes. 
\end{abstract}

\vspace{0.5cm}

{\tiny\noindent\emph{MSC 2000:} primary 37D20; 37C20; secondary 37C27, 34D30.\\
\emph{Keywords:} Reversing symmetry, area-preserving map, closing lemma, elliptic point.\\}

\vspace{0.5cm}

\section{Introduction}\label{intro}

\subsection{Symplectic and reversing symmetry invariants} 
Symmetries, like equivariance and reversibility, play an important role in determining the behavior of a dynamical system. 
Equivariant symmetries have been studied extensively in connection with bifurcation theory; see \cite{FMN, Golubitsky II}. 
The general theory for reversible symmetries in dynamical systems had a more recent development as described in the  thorough surveys \cite{LR, RQ}. Indeed, symmetries are geometric invariants which play an important role through several applications in Physics, from the Classical \cite{Birkhoff} and Quantum Mechanics \cite{Prigogine} to Thermodynamics \cite{Kumicak}. However, there is still a gap in the literature concerning the area-preserving reversible systems and the question about the validity of classic results in the this setting arises.

The main object of the present paper is twofold: we intend to study dynamical systems which keep invariant a reversing symmetry and a symplectic form at the same time. As reported in \cite{LR}, the interconnection between these two types of geometrical invariants is much common since there is a large class of Hamiltonians which come equipped with a reversing symmetry. A paradigmatic example relies on the Chirikov-Taylor standard map \cite{Meiss}. In the present article,  after a scrupulous examination of properly chosen examples, we would like to systematize their features and consider the general family of area-preserving maps which display a reversing symmetry. The area-preserving maps appear naturally when considering time-one maps of a Hamiltonian with $2$ degrees of freedom which is the base of Celestial and Classical Mechanics. In what follows, we characterize the maps that we are going to consider along the paper.

Let $M$ denote a compact, connected, boundaryless, $C^\infty$ Riemannian two-di\-men\-sio\-nal manifold, let $\omega$ be a smooth symplectic form on $M$ and $\mu$ the measure associated to $\omega$ that we call \emph{Lebesgue measure} or \emph{area}. Setting by $ \text{Diff}^{1}(M)$ the set of $C^1$ diffeomorphisms on $M$, let $\text{Diff}^{~1}_{\mu}(M)\subset \text{Diff}^{1}(M)$ stand for the set of $C^1$ diffeomorphisms on $M$ such that if $f \in \text{Diff}^{1}_\mu(M)$, then for every borelian subset $A$ of $M$, the following property holds: $\mu(f^{-1}(A))=\mu(A)$. These maps are often called \emph{area-preserving} diffeomorphisms on $M$. A diffeomorphism $f:M \rightarrow M$ is Anosov if $M$ is a uniformly hyperbolic set for $f$, \emph{i.e.} there are $m \in \mathbb{N}$ and $\sigma\in(0,1)$ such that, for every $x \in M$, there is a $Df$-invariant continuous splitting $T_x M = E_x^u \oplus E_x^s$ such that $\|Df_x^m|_{E_x^s}\|\leq \sigma$ and $\|(Df_x^m)^{-1}|_{E_x^u}\|\leq \sigma$.

We are considering the $C^1$-Whitney topology where a property is said $C^1$-\emph{generic} if it holds on a $C^1$-residual set. A $C^1$-\emph{residual set} is a countable intersection of $C^1$-open and $C^1$-dense sets. Observe that both sets  $\text{Diff}^{1}(M)$ and $\text{Diff}^{~1}_{\mu}(M)$ endowed with the $C^1$-topology are Baire spaces \cite[Sect. 3.5]{Webster}: so every residual set is dense.

In this paper we address the existence of elliptic periodic orbits far from the Anosov maps within the space $\text{Diff}^{~1}_{\mu}(M)$ which exhibit some reversing symmetry. Roughly speaking, a reversing symmetry is a diffeomorphism $R:M \rightarrow M$ such that for all $x \in M$, $DR_x \in \SL(2, \mathbb{R})$ and $R\circ R = Id_M$. A diffeomorphism $f :M \rightarrow M$ is
called $R$-\emph{reversible} if there exists an isometry $R : M \rightarrow M$, which conjugates $f$ to its
inverse $f^{-1}$, \emph{i.e.}, such that the following equality holds:
\begin{equation}
\label{reversivel1}
R\circ f = f^{-1}\circ R.
\end{equation}

 The set of diffeomorphisms that are area-preserving and $R$-reversible will be denoted by  $\text{Diff}^{~1}_{\mu,R}(M)$. The fixed-point subspace of the involution $R$, defined as $Fix(R)=\{x\in M: R(x)=x\}$, is a central object of analysis. While fixed point subspaces of preserving symmetries are invariant under the dynamics, in general they are not setwise invariant for reversing symmetries. 
Throughout this article, this set is assumed to be a smooth one-dimensional submanifold of $M$. 

\medbreak

	\subsection{Elliptic closed points} 
Let $f \in \text{Diff}^{~1}(M)$. We say that $p\in M$ is an \emph{elliptic periodic orbit} of period $n \in \mathbb{N}$ for the diffeomorphism $f$ if the following conditions hold:
\begin{itemize}
\item $f^n(p)=p$ ($p\in Per(f)$ i.e. $p$ is periodic for $f$),
\item $f^i(p)\not=p$, for all $i=1,...,n-1$ and 
\item the map $Df^n_p \in \SL(2, \mathbb{R}) $ has non-real spectrum with modulus $1$.
\end{itemize}
\medbreak

 Starting with Birkhoff \cite{Birkhoff}, the dynamic structure of symplectic maps near an elliptic periodic orbit has been studied intensively in the context of Hamiltonian mechanics. Going back in to the seventies, in the generic theory of area-preserving diffeomorphisms, we find the seminal result of Newhouse  \cite{N} that says that $C^1$-generic diffeomorphisms in surfaces are Anosov or else the elliptical points are dense. The proof
is supported in the symplectic structure and on the finding of homoclinic tangencies associated to periodic orbits. The extension of the Newhouse dichotomy for the reversible context is not so straightforward as we may expect due to the technical difficulties that the reversibility entails: there is no guarantee that the construction of the $R$-reversible and conservative perturbation, through homoclinic tangencies, may be done in a reversible way. This is why our proof does not run along the same lines to that of \cite{N} but rather to those of \cite{BesDu07, BCR}.
\medbreak

The present article gives a proof of  Newhouse's dichotomy in the reversible and area-preserving setting; we also obtain the abundance of elliptic closed orbits $C^1$-far from the Anosov maps and the proof of the stability conjecture for this class. A consequence of the latter result is that the only symplectic reversible diffeomorphisms on a compact manifold which are structurally stable (among symplectic reversible diffeomorphisms) are the Anosov diffeomorphisms. Anosov diffeomorphisms within the class $\text{Diff}^{~1}_{\mu,R}(M)$ were studied with detail in ~\cite[\S4]{BCR}. In the dissipative context, using the concept of heteroclinic tangencies, the authors of \cite{LS} establish a new type of Newhouse domains, in which there is a dense set of diffeomorphisms having simultaneously sinks, sources and elliptic periodic orbits.

\subsection{Statement of the main results}

In order to extend the dichotomy of \cite{N} for area-preserving reversible maps, we first state and prove a new version of the $C^1$-Closing Lemma, one of  the classic results in dynamical systems. This is meant to refer to a bifurcation problem in which there is a recurrent (or non-wandering) non-closed orbit. Recall that a point $x\in M$ is \emph{non-wandering} for $f$ if for all neighbourhood $U$ of $x$, there exists $n \in \mathbb{N}$ such that $f^n(U) \cap U \neq \emptyset$. 
\medbreak

 By perturbing the original system, we obtain a $C^1$-near system that has a periodic orbit passing near the non-wandering point. Several papers deal with the classic $C^1$-Closing Lemma and improvements thereof. We refer the reader to the survey \cite{AZ} that consider numerous kinds of closing lemmas. Using similar arguments to those of \cite{PR} combined with \cite{BCR}, we state and prove the following version of the $C^1$-Closing Lemma:

\begin{maintheorem}[The Reversible $C^1$-Closing Lemma]
\label{Closing}
Let $R$ be an isometric involution on $M$. There exists a residual set $\mathscr{D}\subset \text{Diff}^{~1}_{\mu, R}(M)$ such that if $f \in \mathscr{D}$, then for Lebesgue almost every point $x\in M$, $r>0$ and $\epsilon>0$, there exists $g\in\mathscr{D}$  such that $g$ is $\epsilon$-$C^1$-close to $f$ and $y\in Per(g)$ for some $y\in B(x,r)$.
\end{maintheorem}
\medbreak
Using some perturbation lemmas that will be developed in Section \ref{Perturbation}, the proof of Theorem \ref{Closing} will be addressed in Section \ref{closinglemma}. A direct corollary of the Reversible $C^1$-Closing Lemma together and the Poincar\'e Recurrence Theorem is the following:

\begin{corollary}
(\emph{General Density Theorem})\label{gdt}
There exists a residual subset $\mathscr{P}\subset\text{Diff}^{~1}_{\mu, R}(M)$ such that the closure of the set of hyperbolic or elliptic periodic points of any $f\in\mathscr{P}$ is the whole manifold $M$.
\end{corollary}
\medbreak

Let $f \in \text{Diff}^{~1}_{\mu,R}(M)$. Extending the Bochi-Ma\~n\'e Theorem (\cite{Bo}), the authors of \cite{BCR} proved that there exists a $C^1$- residual subset  $\text{Diff}^{~1}_{\mu,R} (M)$, say $\mathcal{R}$, such that every $f\in \mathcal{R}$ either is Anosov or have zero Lyapunov exponents at Lebesgue almost every point. If $f$ is non-Anosov, the result says nothing with respect to the existence of elliptic periodic points. In this paper, we are able to prove that if $f$ is far from being Anosov, thus for any open set $O\subset M$ and $\epsilon>0$, there exists $g\in\text{Diff}^{~1}_{\mu, R}(M)$ such that $g$ is $\epsilon$-$C^1$-close to $f$ and $g$ has an elliptic periodic orbit through $O$. This is the content of the following reversible version of Newhouse's dichotomy~(\cite[Theorem 1.1]{N}):

\begin{proposition}
\label{newhouse}
If $f\in\text{Diff}^{~1}_{\mu, R}(M)$ be a map in the $C^1$-interior of the complement of the Anosov maps, then for any (non-empty) open set $O\subset M$ and $\epsilon>0$, there exists $g\in\text{Diff}^{~1}_{\mu, R}(M)$ such that $g$ is $\epsilon$-$C^1$-close to $f$ and $g$ has an elliptic periodic orbit through $O$.
\end{proposition}

An interesting consequence of Proposition~\ref{newhouse} relies on the denseness of elliptic periodic orbits:

\begin{maintheorem}\label{newhouse2}
There exists a $C^1$-residual $\mathscr{R}\subset\text{Diff}^{~1}_{\mu, R}(M)$ such that any $f\in \text{Diff}^{~1}_{\mu,R}(M)$ is Anosov or else the elliptic periodic orbits of $f$ are dense in $M$. 
\end{maintheorem}

Recalling that the only surface that admits Anosov diffeomorphisms is the torus $\mathbb{T}^2=\mathbb{R}^2 / \mathbb{Z}^2$ (see \cite{F2}), when $M \neq \mathbb{T}^2$, the previous result can be stated in a more powerful way:
\begin{corollary}
If $M \neq \mathbb{T}^2$, then there exists a $C^1$-residual $\mathscr{R}\subset\text{Diff}^{~1}_{\mu, R}(M)$ such that any $f\in \text{Diff}^{~1}_{\mu,R}(M)$ the elliptic periodic orbits of $f$ are dense in $M$.
\end{corollary}

We also obtain the proof of the stability conjecture for area-preserving reversible maps. Recall, that $f\in \text{Diff}^{~1}_{\mu,R}(M)$ is said to be $C^1$-\emph{structurally stable} if there exists a $C^1$-neighbourhood $\mathcal{U}\subset \text{Diff}^{~1}_{\mu,R}(M)$ such that, for any $g\in\mathcal{U}$, there exists a homeomorphism $h\colon M\rightarrow M$ (not necessarily in $\text{Diff}^{~1}_{\mu,R}(M)$) such that $f\circ h=h\circ g$.

\begin{maintheorem}\label{SC}
A map $f\in \text{Diff}^{~1}_{\mu,R}(M)$ is $C^1$-structurally stable if and only if $f$ is Anosov.
\end{maintheorem}

More precise statements of the results will be given throughout the article.
\medbreak

 Since there is a gap in the classic literature concerning the reversible systems, in the present paper we revisit the construction of perturbation toolboxes, namely local perturbations lemmas \cite{BCR} and the Franks-type lemma \cite{BCR} in Section \ref{Perturbation}, after having provided some preliminaries and notation in Section \ref{preliminaries}. The proof of Theorem \ref{Closing} is done is Section \ref{closinglemma} and those of Proposition \ref{newhouse}, Theorems \ref{newhouse2} and \ref{SC} are addressed in Section \ref{mains}. We also revisit Newhouse's proof ~\cite{N} that obtained a dense set of elliptic points via the existence of homoclinic tangencies of the invariant manifolds of periodic orbits. At the end of this section, the reader will realize why the proof of Proposition \ref{newhouse} does not run along the same lines to that of \cite{N}. 

\section{Preliminaries}
\label{preliminaries}
In this section, we describe precisely the ambient space and properties of the maps under consideration in the article. 
We introduce some general features and the property of reversibility for a map on the surface $M$. 

\subsection{Generalities}\label{Define}
We will use the canonical norm of a bounded linear map $A$ given by $\|A\|=\sup_{\|v\|=1}\|A\, v\|$. By Darboux's theorem (see e.g.~\cite[Theorem 1.18]{MZ}) there exists an atlas $\{\varphi_j\colon V_j\to\mathbb{R}^{2}\}$, where $V_j$ is an open subset of $M$, satisfying $\varphi_j^*\omega_0=\omega$ with $\omega_0= dx\land dy$. Since $M$ is compact we can always take an atlas with a finite number of charts, say $k$. We fix once and for all the finite atlas $\{\varphi_j\colon V_j\to\mathbb{R}^{2}\}_{j=1}^k$. Given $x\in V_j$, for some $j\in\{1,...,k\}$, let $B(\varphi_j(x),r)$ stands for the open ball of radius $r$ centered in $\varphi_j(x)$. Clearly, if $r$ is small enough we get that $\varphi^{-1}_j(B(\varphi_j(x),r))\subset V_j$. In this case we  will refer to $B(x,r)$ the set $\varphi^{-1}_j(B(\varphi_j(x),r))$. 
We notice that all the estimates along the present paper are made in
using the Darboux coordinate charts (\cite{MZ}). Hence, we may assume that $M = \mathbb{R}^{2}$, $T_x M = \mathbb{R}^{2}$, and the exponential map 
$exp_x\colon T_x M \rightarrow M$, which do not preserve necessarily a symplectic structure, is, under these assumptions, the identity.

\subsection{Reversibility}
Here, we introduce basic definitions and recall some results on symmetric systems, whose proof can be found in \cite{LR,RQ}, where it has been presented a compact survey of the literature on reversible dynamical systems.  We say that the diffeomorphism $R\colon M\rightarrow M$ has $n$-degree if $R^n =Id$, where $n\in\mathbb{N}$ is the lower positive integer that satisfy the previous equality and $Id$ the identity map. In the case $n=2$ we have that $R$ is an \emph{involution}, i.e., $R^2=R\circ R=Id$. We say that $R$ is a \emph{reversing symmetry} of a diffeomorphism $f\colon M\rightarrow M$ if we have: 
\begin{equation}\label{RS}
R\circ f = f^{-1}\circ R.
\end{equation} 

\medbreak

Identity (\ref{reversivel1}) implies that if $\mathcal{O}(x)$ is an orbit of $x\in M$, then $R(O(x))$ is also an orbit  where the time is reverted. In other words, a reversing symmetry maps orbits into orbits reversing time direction. 
The definition of reversibility is  algebraic and does not rely on any differentiability property of the diffeomorphism $f$. Note that saying that $f$ is $R$-reversible
 means that $R$ conjugates $f$ and $f^{-1}$. 
 \medbreak
 The reversing symmetries that we are going to consider along the paper will be \emph{isometries} meaning that given the metric $\textbf{g}$ related to the Riemannian structure in $M$ we have $\textbf{g}=R^*\textbf{g}$, that is, the pull-back of $\textbf{g}$, via $R$,  leave the metric $\textbf{g}$ invariant. Since the tangent map of $R$ is a linear isometry it is clear that $R\in\text{Diff}^{1}_{\mu}(M)$. 
In addition, we also assume that the closed subset of fixed points of $R$, that is $Fix(R):=\{x \in M: R(x)=x\}$, has dimension equal to one as is the most common in several examples of the literature -- see for instance \cite{LR} and references therein.

\section{Perturbation Lemmas}
\label{Perturbation}
In the present section, we revisit certain general perturbation statements that are at the basis of most results referred in the article. These perturbation lemmas have been proved in the $C^1$-topology. 
\medbreak
More precisely, we state two useful perturbation lemmas that are the basis of most results mentioned in the text: the first one, adapted from \cite{BCR}, gives a way to perform local perturbations among reversible systems. The second relies on the Franks Lemma \cite{F}, allowing to realize locally small abstract perturbations of the derivative along periodic or non-periodic orbits. The two types of results combine perturbations in the area-preserving setting that leave a reversing symmetry invariant.
\medbreak

The next result shows how to perform a local perturbation in the world of reversible systems in order to keep the perturbations inside this setting. Two general assumptions about the point $x\in M$ around where we are doing a $R$-reversible perturbation are required: $f(x)\neq R(x)$ and $x\notin Fix(R)$. Later on Proposition \ref{D} we will se that these conditions are generic in our context.

\begin{lemma}(\cite[Lemma 7.1]{BCR})\label{local}
Given $f \in \text{Diff}^{~1}_{\mu,R}(M)$ and $\eta>0$, there exist $\rho>0$ and $\zeta>0$ such that, for any point $x\in M$, whose orbit by $f$ is not periodic and $f(x)\not= R(x)$, and every $C^1$ area-preserving diffeomorphism $h\colon M\rightarrow M$, coinciding with the Identity in $M\backslash B(x,\rho)$ and $\zeta$-$C^1$-close to the Identity, there exists $g\in\text{Diff}^{~1}_{\mu, R}(M)$ which is $\eta$-$C^1$-close to $f$ and such that $g = f$ outside $C$ and $g=f\circ h$ in $B(x,\rho)$.
\end{lemma}

\begin{remark}\label{local2}
The local perturbation $h$ can be done in such a way that we get $g\in\text{Diff}^{~1}_{\mu,R}(M)$ satisfying $g=h\circ f$. In fact, by Lemma~\ref{local} we can build $\tilde{g}=f^{-1}\circ h^{-1}$ supported in a ball $\tilde{\mathscr{B}}$ and follow the same idea. Finally, we obtain $g:=\tilde{g}^{-1}=h\circ f$ such that $g=f$ outside $\tilde{\mathscr{B}}\cup f(R(\tilde{\mathscr{B}}))$ and $g=h\circ f$ in $\tilde{\mathscr{B}}$.
\end{remark}

Another perturbation lemma which is a key step in order to obtain several results in dynamics is a perturbation which provides an abstract tangent action by a given map $C^1$-close to an initial one, and was first considered in \cite{F} for dissipative diffeomorphisms as a contribution to the solution of the \emph{Stability Conjecture}; it will also be used to prove Theorem~\ref{SC}. 
\medbreak

We  will consider a finite segment of a given orbit and a small perturbation of the derivative along that segment. Then, we ask if there exists a dynamical system close to the initial one such that its derivative equals the perturbations we considered. 
 
 The following  result is the version of Franks' lemma (see~\cite[Lemma 1.1]{F}) for reversing symmetric diffeomorphisms. Since $f$ is area-preserving, we make use of the conservative Franks' lemma proved in \cite{BDP}. As in \cite{BCR}, we should impose some restrictions.

\begin{definition}
\label{Free2}
Given a subset $X$ of $M$, we say that $X$ is $(R,f)$-free if: 
\begin{equation}
\label{free}
\text{for all}\,\, x,y \in X\,\,\text{we have}\,\,f(x)\not=R(y).
\end{equation}
\end{definition}

It has been proved in \cite[Lemma 3.4]{BCR} that if $f\in \text{Diff}^{~1}_{\mu, R}(M)$, $x\in M$ and $R(x)$ does not belong to the $f$-orbit of $x$, then this orbit is $(R,f)$-free. Thom Transversality Theorem allows us to conclude that:

\begin{proposition} \label{D}
There exists a residual $\mathscr{D}\subset \text{Diff}^{~1}_{\mu, R}(M)$ such that for any $f\in \mathscr{D}$ the set of orbits outside $Fix(R)$ which are not $(R,f)$-free is countable.
\end{proposition}

See the proof in \cite{BCR}. As a trivial conclusion of the previous result and the fact that $\dim(Fix(R))=1$ we obtain that generically in $ \text{Diff}^{~1}_{\mu, R}(M)$ the set of $(R,f)$-free orbits has full Lebesgue measure.
\medbreak
We will now consider an area-preserving reversible diffeomorphism, a
finite set in $M$ and an abstract tangent action that performs a small perturbation of the derivative along that
set. Then we will search for an area-preserving reversible diffeomorphism, $C^1$ close to the initial one, whose
derivative equals the perturbed cocycle on those iterates. To find such a perturbed diffeomorphism, we will
benefit from the argument, suitable for area-preserving systems, presented in \cite{BDP}.

\begin{lemma}\cite[Lemma 7.4]{BCR}\label{Franks}
Fix an isometric involution $R$ and $f\in\text{Diff}^{~1}_{\mu,R}(M)$. Let $\Theta:=\{x_1,x_2,...,x_k\}$ be a finite set of distinct points in $M$. Asume that $\Theta$ is $(R,f)$-free. Denote by $Q=\oplus_{x\in \Theta} T_x M$ and $Q'=\oplus_{x\in \Theta} T_{f(x)} M$. Let $G\colon Q\rightarrow Q'$ be a linear area-preserving map. For every $\epsilon>0$, there exists $\delta>0$, such that if $\|G-Df\|<\delta$, then there exists $g\in\text{Diff}^{~1}_{\mu,R}(M)$ $\epsilon$-$C^1$-close to $f$ and satisfying $Dg_x=G|_{T_x M}$ for every $x\in\Theta$. Moreover, if $K\subset M$ is a compact on $M$ and $K\cap\Theta=\emptyset$ then $g$ can be chosen such that $g=f$ in $K$.
\end{lemma}

\medbreak

The following result is the peak of the well known Ma\~n\'e dichotomies on periodic orbits - the dominated splitting is the only obstruction to obtain trivial spectrum on periodic orbits. The abstract general result for dissipative systems was obtained in \cite[Corollary 2.19]{BGV}. Here we consider the area-preserving reversible version.

\begin{theorem}
\label{mainBGV} Let $f\in\text{Diff}^{~1}_{\mu,R}(M)$. Then for any $\epsilon>0$ there are two integers $m$ and $n$ such that, for any periodic
point $x$ of period $p(x) \geq n$:
\begin{enumerate}
\item the orbit of $x$ is an $(R,f)$-free set or either;
\item $f$ admits an $m$-uniform hyperbolic splitting along the orbit of $x$ or else;
\item for any neighbourhood $U$ of the orbit of $x$, there exists $g\in\text{Diff}^{~1}_{\mu,R}(M)$ $\epsilon$-$C^1$-close to $f$, coinciding with $f$ outside $U$ and on the orbit of $x$, and such that
$x$ is a parabolic point of $g$ for which the differential $Dg^{p(x)}_x$ has all eigenvalues
real and with the same modulus, thus equal to $\pm1$.
\end{enumerate}
\end{theorem}

The proof of Theorem~\ref{mainBGV} follows closely the steps of the work ~\cite{BGV}, this is why we shall omit it. However, we point out the key points that we have to be careful about:
\begin{enumerate}
\item [(a)] Firstly, in ~\cite{BGV}, the authors need not be worried about perturbation obstructions as the ones we are considering due to the reversibility: the periodic orbits in their perturbations need not to be $(R,f)$-free. In our case, since the Lebesgue measure cannot ``see'' the set of the orbits that are not $(R,f)$-free, we may exclude them of the scenario.
\medbreak
\item [(b)] Secondly, it is worth to stress that the item (2) of Theorem~\ref{mainBGV} refers to obtaining dominated splitting. By \cite[Lemma 3.11]{Bo}, in a two-dimensional area-preserving setting, $m$-uniform hyperbolicity is equivalent to $m$-dominated splitting.
\medbreak
\item [(c)] Finally, in the absence of uniform hyperbolicity, we will combine the arguments in ~\cite{BGV} with Lemma~\ref{Franks},
when perturbations are requested to obtain (3). The perturbations that are used in \cite{BGV} are rotations which are clearly area-preserving (see also \cite[Lemme 6.6]{BC} performed for $\SL(2,\mathbb{R})$ cocycles). These transformations are used to obtain real spectrum with the same modulus (\emph{i.e.} $\pm1$). 
\end{enumerate}

\bigskip

If $\mathscr{D}\subset \text{Diff}^{~1}_{\mu, R}(M)$ is the $C^1$-residual set of Proposition~\ref{D}, the  Kupka-Smale theorem in $ \text{Diff}^{~1}_{\mu, R}(M)$ allows us to conclude that:

\begin{lemma}
There exists a $C^1$-residual $\mathscr{D}\subset \text{Diff}^{~1}_{\mu, R}(M)$ such that for any $f\in \mathscr{D}$ the set of periodic orbits satisfying (2) or (3) of Theorem~\ref{mainBGV} are dense in $Per(f)$.
\end{lemma}

If $\mathscr{Q}= \mathscr{D}\cap \mathscr{P}$, where $\mathscr{P}$ is the $C^1$-residual given in Corollary~\ref{gdt}, we easily obtain:

\begin{lemma}\label{good}
There exists a $C^1$-residual $\mathscr{Q}\subset \text{Diff}^{~1}_{\mu, R}(M)$ such that for any $f\in \mathscr{Q}$ the set of periodic orbits satisfying (2) or (3) of Theorem~\ref{mainBGV} are dense in $M$.
\end{lemma}

\section{The Reversible $C^1$-Closing Lemma}
\label{closinglemma}

A fundamental result on dynamical systems, which goes back to Poin\-car\'e, is the well known \emph{Closing Lemma}. In rough terms, we intend to \emph{close}, in a sense that we turn it into a periodic orbit, a given recurrent or non-wandering orbit by making a small perturbation on the original system. Until now, there are satisfactory answers to this problem if the small perturbations are with respect to lower topologies ($C^0$ and $C^1$). Despite the fact that the $C^0$-closing lemma is a quite simple exercise, the $C^1$ statement reveals several difficulties. The $C^1$-closing lemma was first established by Pugh \cite{P,P2} in the late 1960's and, in the early 1980's, by Pugh and Robinson \cite{PR} for a large class of systems as volume-preserving, symplectic diffeomorphisms as also for Hamiltonians. In what follows, we present the reversing symmetric version of it. But before that we will introduce two core tools developed in ~\cite{PR}. 
 \medbreak
\begin{enumerate}
\item [(A)] A fundamental key step to obtain the closing of trajectories is the \emph{Lift Lemma}, proved for area-preserving diffeomorphisms in \cite{PR} (see \S 8 b) or a) because our area-preserving context is both symplectic and volume-preserving). If very brief terms, \emph{lifting} implies that we can push points in a given direction, acting only locally and supported in a given small ball. The intensity of this perturbation is uniform from point to point and also proportional to the size of the ball. Due to our area-preserving and reversing symmetric restrictions the previous push must be done within this class. The area-preserving push is guaranteed by Pugh and Robinson judicious  perturbations. To treat the $R$-invariance of the perturbations adapted very carefully the ideas of ~\cite{PR}. 
\medbreak
\item [(B)] The other tool has more to do with the drawbacks intermediate recurrences and how we can use them in our favor and also to obtain some \emph{conformal} behavior. In fact, given $y$ and its iterate $f^n(y)$, very close to one another, if we intend to push $y$ to $f^{n}(y)$ by a small $C^1$-perturbation, two kind of problem occur: first a $C^1$-perturbation may not allow a sufficiently large shift as we have seen in (A), and second, nothing assures us that we have a clean scenario in their neighbourhood, say free from  intermediate recurrences formed by points $f^{i}(y)$ with $i=1,...,n-1$. Fortunatly, the ingenious ideas developed in \cite{PR} solve this hard obstructions. For that we have the \emph{Fundamental Lemma} in \cite[\S4]{PR}, which assures that we can always find two intermediate isolated recurrences:
\begin{itemize}
\item $f^{i}(y)$ and $f^{j}(y)$ with $i<j$ inside a box $B$ and 
\item if we inflate $B$ a little bit there are no elements in $f^{k}(y)$ for $k=\{0,...,n\}\setminus \{i,j\}$.
\end{itemize}
\end{enumerate}

\medbreak

It has been proved in \cite{BCR} that if $f\in \text{Diff}^{~1}_{\mu, R}(M)$, $x\in M$ and $\mathcal{O}(x)\cap \{R(x)\}=\emptyset$ then, the set $\mathcal{O}(x)$ is $(R,f)$-free. Recall Proposition \ref{D} that says that generically in $ \text{Diff}^{~1}_{\mu, R}(M)$ the set of $(R,f)$-free orbits has full Lebesgue measure. The proof of Theorem \ref{Closing} will be done in several steps, using arguments of ~\cite{PR}. We suggest that the reader follows the proof observing Figure \ref{Closing Lemma 1}.

\begin{proof}
Let $R$ be an isometric involution and $f \in \mathscr{D}\subset \text{Diff}^{~1}_{\mu, R}(M)$ where $\mathscr{D}$ is the residual set given Proposition~\ref{D}. Following Pugh and Robinson \cite{PR}, we will perform an $\epsilon$-perturbation of $f$ supported on small disks that will be specified later. We recall the \emph{Axiom Lift Lemma} proved in \cite{PR} that says that, in the pattern given by a basis $E$ of $T M$, one can shift the orbit through $p$ on the orbit through $\tilde{p}$ in $N$ steps, without changing $f$ out of the $N$  first iterates of the ``square'' which is $(1+\eta)$-times the smallest square containing both $p$ and $\tilde{p}$ with $\eta > 0$ very small. The proof of the Axiom Lift is related to the \emph{Selection Theorem} \cite[Th. (3.2)]{PR}.
\medbreak
The key to the proof of the $C^1$-Closing Lemma is the following technical result on linear algebra (on the tangent bundle). It combines two crucial tools namely the Selection theorem and the Fundamental lemma (\cite[Lemma (4.1)]{PR}).

\medbreak

\begin{lemma}\label{crucial}
Given $\epsilon>0$, $\eta>0$  and $K>1$  there is $N \in \mathbb{N}$  such that: for any set of linear maps $A_0, \ldots, A_{N-1}$ in $\SL(\mathbb{R}^2)$ satisfying $\|A_i ^{\pm1}\|<K $, there is an orthogonal basis $E=(e_1, e_2)$  of $\mathbb{R}^2$  such that, for every pair of points $a,b$ in the unitary square with respect to $E$ ($\mathcal{Q}=[-1, 1]_E^2$), there is a sequence $g_i$,  $i=0,\ldots,N-1$,  of area-preserving diffeomorphisms $g_i\colon \mathbb{R}^2_i\rightarrow \mathbb{R}^2_{i+1}$ with:
\begin{enumerate}
\item $\|g_i- A_i\|<\epsilon$;
\item $g_i$ is equal to $A_i$ out of the image of $A_{i-1} \circ A_{i-2} \circ \ldots \circ  A_0({[-1-\eta, 1+\eta]_E^2})$ and
\item $g_{n-1} \circ g_{n-2} \circ \ldots g_0(a)=A_{n-1} \circ A_{n-2} \circ \ldots A_0(b)$.
\end{enumerate}
\end{lemma}
\medbreak
\medbreak

Before we continue the proof, it is worth to interpret the meaning of the previous powerful result. Condition (2) means that, up to a dimensional constant, we can specify the proportions of the square $\mathcal{Q}$ with respect to the basis $E$ inside which a ``gradual'' perturbation occurs. 
\medbreak

The performed perturbations occur in the tangent bundle. To ``go down'' to the manifold $M$, we need to realize the perturbation on $M$. This is the goal of the Lift Axiom Lemma (\cite[Def. pp. 265]{PR}), which ensures that for each $ f\in\text{Diff}^{~1}_\mu (M)$ and each $C^1$ neighbourhood $\mathcal{N}$
of $f$, there exists a uniform $\sigma > 0$ such that for all $p \in M$, $v \in T_p M$  with $||v||_E=1$, we have an area-preserving perturbation 
$h$ of the identity satisfying the following properties:
\begin{itemize}
\item $h \circ f \in \mathcal{N}$;
\item $h(p)=exp_p(\sigma v): T_p M \rightarrow M$ and 
\item the topological closure of the set where $h$ differs from the
identity is contained in the ball centered on $p$ and radius $||v||$. 
\end{itemize}
This means that the perturbation $h$ can lift points $p$ in a prescribed direction $v$ with results $\sigma$-proportional to the support. 
  
\begin{figure}\begin{center}
\includegraphics[height=15cm]{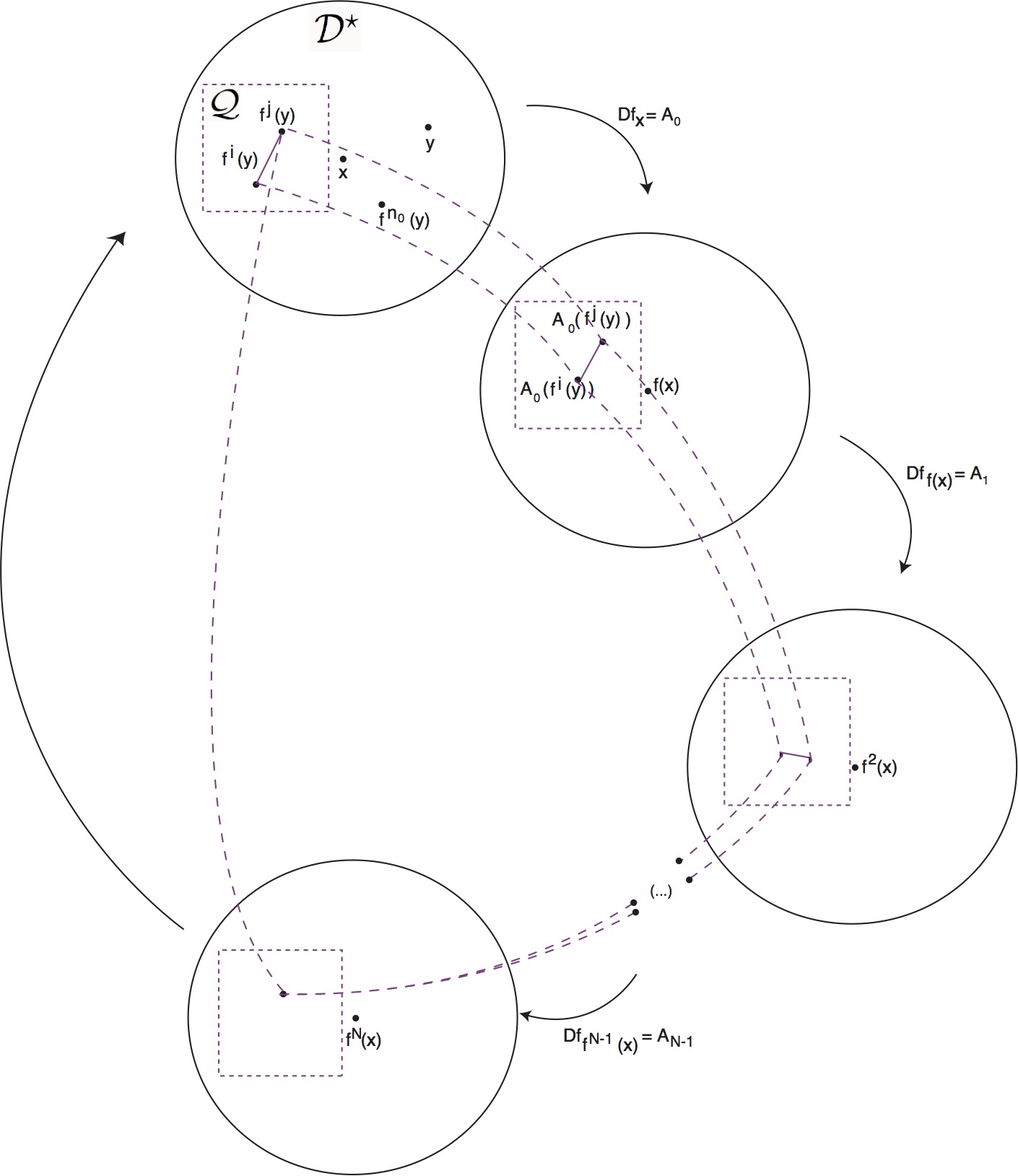}
\end{center}

\caption{There is a finite set of points $y_i$, $i \in \Lambda$, in the disk $\mathscr{D}$.  \emph{Fundamental Lemma} shows that there exists two of these points $f^i(y)$ and $f^j(y)$, $i<j<n_0$, and a square $\mathcal{Q}$  containing them such that the homothetic square $(1+\eta)\mathcal{Q}$ does not contain any other point $f^k(y)$ where $k \in \{1, \ldots,n_0\}\backslash \{i,j\}$.}
\label{Closing Lemma 1}
\end{figure}

Let us fix the size $\epsilon>0$ of the perturbations we allow and pick some small numbers $\eta>0$ and $K>1$. Remark~\ref{local2} allows to fix, once and for all, an integer $N$ such that:
$$
N>\frac{40\mathpzc{B}}{\sigma},
$$
where $\mathpzc{B}$ is a uniform bound for the global distortion along $Df_{f^i(p)}$, $i=1, \ldots, N$ where, for all $i \in \{1,\ldots, N\}$ we have $\|Df_{f^i(p)}^{\pm1}\|<K$; the constant $\mathpzc{B}\geq 1$ is what the authors of \cite{PR} call the \emph{altitude bolicities} of $Df_{f^i(p)}$ for $i=0, \ldots, N-1$.
\medbreak

Let $x \in \Omega\backslash Per(f)$ be a non-wandering point with respect to $f$. By definition of non-wandering point, there exists $y\in M$ arbitrarily close to $x \in \Omega$, and $n_0 \in \mathbb{N}$ such that
$n_0>N$ and the iterate $f^{n_0}(y)$ is arbitrarily close to $x$. The points $y$  and $f^{n_0} (y)$ are so close to $x$ that we may assume that the restrictions of $f^i$ to $\mathcal{D}^\star=D(x,r_1)$ is governed by the linear map $Df$ along the $f$-iterates of $\mathscr{D}$ (if necessary, shrink the neighbourhood of $x$), obeying the hypothesis of Lemma \ref{local}.
\medbreak
Lemma~\ref{crucial} provides local coordinates on the disk $\mathcal{D}^\star$ given by the basis $E$. Now, we concentrate our attention at the set $\Upsilon=\{0,\dots,n\}$  of all the return times of $y$  in the disk $\mathcal{D}^\star$:
$$
i \in \Upsilon \Leftrightarrow f^i(y) \in \mathcal{D}^\star. 
$$
So we get a finite set of points $y_i$, $i \in \Upsilon$, in the disk $\mathcal{D}^\star$.  Lemma~\ref{crucial} shows that there exists two of these points $f^i(y)$ and $f^j(y)$, $i<j<n_0$, and a square $\mathcal{Q}$  containing them such that the homothetic square $(1+\eta)\mathcal{Q}$ does not contain any other point $f^k(y)$ where $k \in \{1, \ldots,n_0\}\backslash \{i,j\}$.  
\medbreak
Since the orbit is $(R,f)$-free,  it follows that $f(x)\neq R(x)$ and thus the hypothesis of Lemma \ref{local} holds. Using now  Lemmas \ref{crucial} and  \ref{local} (acting together) we may perform two balanced local perturbations in order to obtain a resultant map inside the class of $R$-reversible maps. More precisely, both results build a $\epsilon$-$C^1$-perturbation $$g:=g_{N-1}\circ\ldots\circ g_0$$ of $f$  which is equal to $f$ out of the $N$ first $f$-iterates of the square $(1+\eta)\mathcal{Q}$, and such that 
$$g (f^j(y)) =g(f^i(y)).$$ As $f$  has not been changed on $f^{N+i}(y)$ with $i\geq 0$,  one gets that $f^j(y)$ is a periodic orbit of $g$  of period $j-i$. A ``twin" perturbation is automatically constructed in the following sense: taking into account Lemma \ref{local}, we may define additional perturbations in $\bigcup_{i=0}^{N-1} R(f^i(\mathcal{Q}))$ and one obtains $\hat{g}\in\text{Diff}^{~1}_{\mu, R}(M)$ $\epsilon$-$C^1$-close to $g$ such that $ \hat{g}=g$ outside $\bigcup_{i=0}^{N-1} R(f^i(\mathcal{Q}))$ and $\hat{g}=R\circ g$ inside $\bigcup_{i=0}^{N-1} R(f^i(\mathcal{Q}))$. Please note that we did not overlap the two perturbations during both perturbations because we have started with a small neighbourhood $\mathcal{D}^\star$ of $x$ such that for all $i \in \{0,1, \ldots, N\}$, $f^i(\mathcal{D}^\star)$ does not intersect $R(x)$.

 \end{proof}
\medbreak
By the $C^1$-Closing Lemma, one knows that very non-wandering point can be made periodic by a small $C^1$-perturbation. This periodic point can be made hyperbolic or elliptic by a new perturbation, persisting under small perturbations.

\section{Proof of Proposition \ref{newhouse} and Theorems \ref{newhouse2} and \ref{SC} }
\label{mains}
The main goal of this section is the proof of Proposition \ref{newhouse} and Theorems \ref{newhouse2} and \ref{SC}. We first revisit the Newhouse proof ~\cite{N} that obtained a dense set of elliptic points via the existence of homoclinic tangencies of the invariant manifolds of periodic orbits. 

\subsection{Elliptic points from homoclinic tangencies: revisiting Newhouse's proof}
The approach followed by Newhouse in ~\cite{N} was to obtain elliptic points near homoclinic tangencies, when the stable and the unstable manifolds associated to a hyperbolic periodic orbit intersect in a non-transversal way.
\medbreak

 Let us see how the elliptic points are obtained: extending a result of Zehnder \cite{Z0}, in \cite{N}, it is first proved that if $f$ is non-Anosov, then homoclinic tangencies associated to hyperbolic periodic orbits are created for $g$ $C^1$-close to $f$ -- more details in \cite[pp.1078]{N}. The perturbations are not explicit. The homoclinic tangencies are created near the point where no transversality exists and they are done in such a way that the angle between the stable and the unstable direction of a periodic hyperbolic point has a homoclinic tangency. 
 \medbreak
 
 We point out that, in the dissipative case, an explicit relationship between homoclinic tangencies and the angle of the stable and unstable subspaces of periodic points has been given in \cite[Lemma 2.2.1]{PS}, which cannot be directly extended to our study. Secondly, in \cite[Lemma 4.1]{N}, the author shows that a symplectic diffeomorphism with a non-transverse homoclinic point (for some hyperbolic periodic orbit) can be perturbed to produce an elliptic periodic orbit nearby. For a small $\epsilon>0$ and $N \in \mathbb{N}$, the idea is to construct an isotopy $g_t$ ( $-\epsilon\leq t \leq \epsilon$) of perturbing maps in $\text{Diff}_\mu^1(M)$ such that:
 \begin{itemize}
 \item $g_0 = f$
 \item there is a point $s \in [-\epsilon,\epsilon]$ for which $g^N_{s}$ has a bifurcation (a fixed point appears).
 \end{itemize}
 It implies that $(Dg^N_{s})_p$ has 1 as an eigenvalue of multiplicity two (for the map $Dg^N$) and thus an elliptic periodic point. The combination of these two results constitutes the proof of the dichotomy. 
 
\subsection{Proof of Proposition~\ref{newhouse}} 
In this section, we put together the previous information about and we show Proposition~\ref{newhouse}. Let $f\in\text{Diff}^{~1}_{\mu, R}(M)$ be a map in the $C^1$-interior of the complement of the Anosov maps, $O\subset M$ a non-empty open set and $\epsilon>0$. We will show that there exists $g\in\text{Diff}^{~1}_{\mu, R}(M)$ such that $g$ is $\epsilon$-$C^1$-close to $f$ and $g$ has an elliptic periodic orbit through $O$. 
\medbreak
Before proceeding the proof, we need to recall a basic result on topological dimension. There are several different ways of defining the topological dimension of topological space $A$, which we shall denote by $\dim(A)$. For separable metrizable spaces, all these definitions are equivalent. The topological dimension is a topological invariant. For full details on this concept see ~\cite{HW}, where it is proved that: 

\begin{lemma}(Szpilrajn Theorem \cite{HW})\label{Sz}
Given $A\subset \mathbb{R}^2$. If $A$ has zero Lebesgue measure, then $\dim(A)<2$.
\end{lemma}

In order to show Proposition~\ref{newhouse}, we start by proving the following result:

\begin{theorem}\label{empty}
Given $f\in \text{Diff}^{~1}_{\mu, R}(M)$  and $\Lambda_f\subseteq M$ a uniformly hyperbolic set, then either $\Lambda_f=M$ (Anosov) or else $\Lambda_f$ has empty interior.
\end{theorem}

\begin{proof} 
Let $f\in \text{Diff}^{~1}_{\mu, R}(M)$, $\Lambda_f\subset M$ be a uniformly hyperbolic set and $\Lambda \neq M$ (\emph{i.e.} $f$ is not Anosov). 
We want to prove that $\Lambda_f$ has empty interior. We may assume that $f$ is not on the boundary of Anosov reversible maps on $M$. 

By contradiction, suppose that $\Lambda_f$ has positive interior. This means that there is an open set $\emptyset \neq U_f\subset M$ such that $U_f \subset \Lambda$. Any open set has the topological dimension of the ambient space, thus $\dim(\Lambda_f)=2$. 

Hyperbolic sets have a hyperbolic continuation by any small $C^1$ perturbation of $f$, which is also hyperbolic. Denote by $g$ any small $C^1$ perturbation  of $f$ and $\Lambda_g$ the hyperbolic continuation of $\Lambda_f$ (by $g$). Since the conjugacy maps non-empty open sets into non-empty open sets, then 
the set $\Lambda_g$ should contain a non-empty open set and therefore $\dim(\Lambda_g)=2$. 

By Zehnder \cite[Section 2]{Z} we may smoothing out the map $f$ \emph{i.e.} there exists $g^\star\in\text{Diff}^{~\infty}_{\mu}(M)$ $C^1$-close to $f$.  Notice that there is no need $g^\star$ being $R$-reversible. If $\Lambda_{g^\star}$ is the hyperbolic continuation of $\Lambda_f$, then it is clear that $\dim(\Lambda_{g^\star})=2$. 
Using now \cite[Appendix B]{BV}, either $\mu(\Lambda_{g^\star})=0$ or $\Lambda_{g^\star}=M$. The second case does not hold because $f$ is far from the Anosov maps. Therefore, $\mu(\Lambda_{g^\star})=0$. By Lemma~\ref{Sz} we conclude that $\dim(\Lambda_{g^\star})<2$ which is a contradiction because it should contain an open set.
 
\end{proof}

\begin{proof}(of Proposition~\ref{newhouse})
Let $f\in \text{Diff}^{~1}_{\mu, R}(M)$, let $f$ be $C^1$-far from the Anosov maps and let $O$ be a non-empty open set of $M$ . Through a $C^1$-small perturbation we may assume that 
$f\in \mathscr{Q}$ and $f$ is not Anosov, where $\mathscr{Q}$ is the $C^1$-residual set given in Lemma~\ref{good}. Now, two situations should be taken in consideration:
\begin{enumerate}
\item if some parabolic periodic orbit of $f$ goes through $O$ we are over by just using Lemma~\ref{Franks} to turn it into an elliptic periodic point;
\item otherwise, if a dense subset of $m$-uniformly hyperbolic periodic orbits of $f$ go through $O$, then its closure defines a hyperbolic set $\Lambda\supset O$. Since $f$ is $C^1$-far from the Anosov maps, by Theorem \ref{empty}, we get that $\Lambda$ has empty interior, which is a contradiction.
\end{enumerate}
\end{proof}

\subsection{Proof of Theorem~\ref{newhouse2}}
In order to prove Theorem~\ref{newhouse2}, we strongly make use of Proposition~\ref{newhouse}.
We are going to exhibit a residual set $\mathscr{R}_1$ such that $ \mathscr{R}=\mathscr{A}\cup \mathscr{R}_1$ is a residual set of $\text{Diff}^{~1}_{\mu, R}(M)$ for which the dichotomy of Theorem~\ref{newhouse2} holds. Denote by $$\mathscr{P}=\text{Diff}^{~1}_{\mu, R}(M)\setminus \overline{\mathscr{A}},$$  the $C^1$-open set defined by the complement of the $C^1$-closure of the set of Anosov maps $\mathscr{A}$ in $\text{Diff}^{~1}_{\mu, R}(M)$, where $\overline{A}$ denotes the $C^1$-closure of the set $A$.

  Let $\Phi$ stands for the subset of $\text{Diff}^{~1}_{\mu, R}(M) \times M \times \mathbb{R}^+$ such that $(f, x, \varepsilon) \in \Phi$ if and only if  $f$ has an elliptic periodic orbit  intersecting the open ball $B(x, \varepsilon)$. Observe that if $\mathcal{U}\subset \mathscr{P}$ is an open set, then $\Phi(\mathcal{U}, x, \varepsilon)$ defined by:
$$\Phi(\mathcal{U}, x, \varepsilon):=\{g\in \mathcal{U}\colon (g,x,\epsilon)\in \Phi\}$$  
  is an open set for the product topology of $\text{Diff}^{~1}_{\mu, R}(M) \times M \times \mathbb{R}^+$.

Let $(x_n)_{n \in \mathbb{N}}$ be a dense sequence in $M$ (it exists because $M$ is a compact set) and $(\varepsilon_n)_{n \in \mathbb{N}}$ a sequence of positive real numbers converging to zero.  Now, for $n \in \mathbb{N}$, define $U_1=\mathscr{P}$ and $U_{n+1}=\Phi(U_n, x_n, \varepsilon_n)$. The set $\mathscr{R}_1=\bigcap_{n \in \mathbb{N}}U_n$ is the countable union of open sets and if $f \in \mathscr{R}_1$ then the elliptic periodic orbits of $f$ are dense in $M$. 

\subsection{Proof of Theorem~\ref{SC}}
In order to prove Theorem~\ref{SC}, we use Theorem~\ref{newhouse2}.

($\Leftarrow$) If $f$ is Anosov, it is known that $f$ is $C^1$-structurally stable. ($\Rightarrow$) Suppose, by contradiction, that is $C^1$-structurally stable and non-Anosov. 
By Theorem~\ref{newhouse2} there exists a $C^1$-residual set $\mathscr{R}\subset\text{Diff}^{~1}_{\mu, R}(M)$ such that any $f\in \text{Diff}^{~1}_{\mu,R}(M)$, elliptic periodic orbits of $f$ are dense in $M$. Bearing in mind that an elliptic periodic point $x_0$ on the plane is conjugated to a  rotation around an open neighbourhood $V$ of $x_0$, by the Pasting Lemma proved in \cite[Theorem 3.6]{ArMa}, for any $\varepsilon>0$ we may construct two perturbations of $f$, say $g_1$ and $g_2$, inside $\text{Diff}^{~1}_{\mu, R}(M)$, which are in the $C^1$-domain of topological conjugacy of $f$ (note that we are assuming that $f$ is $C^1$-structurally stable). Therefore, there exist homeomorphisms $h_1,h_2\colon M\rightarrow M$ such that 
$$h_1\circ g_1=f\circ h_1\qquad \text{and} \qquad  h_2\circ g_2=f\circ h_2.$$ Those perturbations can be made such that $g_1$ is a rotation of rational angle centered at $x_0$ and $g_2$ is a rotation of irrational angle centered at $x_0$. Then, there exists $h:=h_1^{-1}\circ h_2$ conjugating $g_1$ and $g_2$, i.e., 
$$h\circ g_2= h_1^{-1}\circ h_2 \circ g_2= h_1^{-1}\circ f \circ h_2= g_1\circ h_1^{-1} \circ h_2=g_1\circ h,$$
which is a contradiction.

\section*{Acknowledgements} 
 MB was partially supported by National Funds through FCT - ``Funda\c{c}\~{a}o para a Ci\^{e}ncia e a Tecnologia", project PEst-OE/MAT/UI0212/2011. CMUP is supported by the European Regional Development Fund through the programme COMPETE
and by the Portuguese Government through the Funda\c{c}\~ao para a Ci\^encia e a Tecnologia (FCT) under the
project PEst-C/MAT/UI0144/2011. AR was supported by the grant SFRH/BPD/84709/2012 of FCT.

\end{document}